\documentclass[12pt, a4paper, reqno, oneside]{amsart}
\usepackage{amssymb}  
\usepackage{amsmath} 
\usepackage{amsthm}  
\usepackage{amsaddr}
\usepackage{url}

\usepackage{cite, verbatim, longtable}

\newtheorem{theorem}{Theorem}
\newtheorem*{problem}{Problem}

\newtheorem{prop}[theorem]{Proposition}

\newtheorem{lemma}{Lemma}
\numberwithin{lemma}{section}
\newtheorem{tab}{Table}

\numberwithin{equation}{section}

\newcommand{\cF}{\mathcal{F}}
\newcommand{\cP}{\mathcal{P}}
\newcommand{\cD}{\mathcal{D}}

\def\ov{\overline}

\begin{document}

\vspace{1cm}

\title[Factoring nonabelian finite groups into two subsets]{Factoring nonabelian finite groups into two subsets}

\author{Ravil Bildanov}
\address{Specialized Educational Scientific Center of Novosibirsk State University\\ Pirogova 11/1, Novosibirsk 630090, Russia\\
ravilbildanov@gmail.com}

\author{Vadim Goryachenko}
\address{Novosibirsk State University\\ Pirogova 1, Novosibirsk 630090, Russia\\
goryachenko-vadim@rambler.ru}

\author{Andrey V. Vasil'ev}
\address{Sobolev Institute of Mathematics\\ Koptyuga 4, Novosibirsk 630090, Russia\\
Novosibirsk State University\\ Pirogova 1, Novosibirsk 630090, Russia\\
vasand@math.nsc.ru}


\thanks{The third author was supported by the program of fundamental scientific researches of the Russian Federation, Project 0314-2019-0001.}

\begin{abstract}
A group $G$ is said to be factorized into subsets $A_1, A_2,$ $\ldots, A_s\subseteq G$ if every element $g$ in $G$ can be uniquely represented as $g=g_1g_2\ldots g_s$, where $g_i\in A_i$, $i=1,2,\ldots,s$. We consider the following conjecture:  for every finite group $G$ and every factorization $n=ab$ of its order, there is a factorization $G=AB$ with $|A|=a$ and $|B|=b$. We show that a minimal counterexample to this conjecture must be a nonabelian simple group and  prove the conjecture for every finite group the nonabelian composition factors of which have orders less than $10\,000$.

 \end{abstract}

\maketitle

\section{Introduction}\label{intro}

A group $G$ is said to be {\em factorized into subsets} $A_1, A_2,\ldots, A_s\subseteq G$ if every $g\in G$ can be uniquely represented as  $g=g_1g_2\ldots g_s$, where $g_i\in A_i,i=1,2,\ldots,s$. If $G$ is finite, then the uniqueness condition is obviously equivalent to the condition $|G|=|A_1||A_2|\ldots|A_s|$.

A problem of factoring abelian groups into subsets arose in connection with a problem of tiling a space in geometry. At the end of the nineteenth century, Minkowski conjectured that a lattice cube tiling of $\mathbb{R}^n$ always contains a two cubes sharing an $(n-1)$-dimensional face \cite{Min}. Forty years later, Haj\'{o}s reformulated Minkowski's question as the problem of factoring finite abelian groups: if a finite abelian group is factoring into cyclic subsets, then one of them must be a subgroup (see \cite{Haj1}); then he solved it in \cite{Haj2}. Since then, the factorization theory of finite abelian groups became an attractive and flourishing field of research, and besides the geometry of tilings, found applications in the code theory, cryptography, graph theory, number theory, Fourier analysis, theory of complex Hadamard matrices and so on (see, e.g., \cite{Sz,SzS}).

Let us turn now to the factorization theory of arbitrary (not necessarily abelian) finite groups.
Apart from rather rare attempts to transfer some results about abelian groups\footnote{It is worth mentioning here the article \cite{Sands}, where the author transferred his positive partial solution of Fuch's problem for abelian groups to the case of nilpotent groups.}, the most attention was paid to the {\em minimal logarithmic signature problem} or, briefly, the {\em MLS-problem}, asking (up to the notation) whether every finite group can be factorized into subsets whose cardinalities are primes or four. This problem was formulated in \cite{GVS} and motivated by a search for most effective ways to define cryptosystems based on factorization of finite groups \cite{Mag}. It was reduced to nonabelian simple groups \cite{GVRS}, so now the progress towards a solution is based on the classification of finite simple groups (CFSG). In fact, the progress is substantial, see, e.g, \cite{GVRS, SSM}, but the complete solution seems far to be achieved.

As observed, the MLS-conjecture holds true for solvable groups. In particular, the alternating group $G=A_4$ obviously has a factorization $G=ABC$, where the ordered tuple $(|A|,|B|,|C|)$ equals $(2,2,3)$. However, as it was recently shown \cite{Ber}, the same group $G$ does not have any factorization $G=ABC$, where $(|A|,|B|,|C|)=(2,3,2)$ (see also \cite{MOB}). It follows that the answer to part~(a) of the following question by M.H.~Hooshmand is negative for $k\geq3$.

\begin{problem}[{see \cite[Problem~19.35]{KN}}]\label{mc}
Let $G$ be a finite group of order $n$.

\emph{a)} Is it true that for every factorization $n=a_1\cdots a_k$ there exist subsets $A_1,\ldots,A_k$ such that $|A_1|=a_1,\ldots, |A_k|=a_k$ and $G=A_1\cdots A_k$\emph{?}

\emph{b)} The same question for the case $k=2$.
\end{problem}

Thus, the remaining challenge is to answer part (b) of the problem. The present paper is a step towards this goal.

Let $\cF$ be the class of finite groups which can be factorized into two subsets for every two factors of their orders. Given a group $G$ of order $n$, let $\cD(G)$ stand for the set of divisors $a$ of $n$ such that $G$ can not be factorized into two subsets of orders $a$ and $n/a$. So $G\in\cF$ if and only if $\cD(G)=\varnothing$. Since for every finite group $G$ having factorization $G=AB$,  $$AB=G=G^{-1}=B^{-1}A^{-1},$$
it follows that $a\in\cD(G)$ if and only if $n/a\in\cD(G)$, in particular, $\cD(G)=\varnothing$ if $a\not\in\cD(G)$ for every $a\leq\sqrt{|G|}$.

The following two propositions are rather simple and the facts they reflect were discussed in one or another way (see~\cite{MOH,MOB}). We include their proofs (see Section~\ref{pth2}), because, as far as we aware, they have never been published.

\begin{prop}\label{mth}
Let $G$ be a finite group of the smallest order such that $G\not\in\cF$. Then the following hold:
\begin{enumerate}
\item $G$ is a nonabelian simple group;
\item if $a\in\cD(G)$, then neither $a$ nor $|G|/a$ divides the order of any proper subgroup of $G$.
\end{enumerate}
\end{prop}

Let $\cP$ be the class of finite groups having a subgroup series
\begin{equation}\label{series}
1=G_0<G_1<\ldots<G_{t-1}<G_t=G
\end{equation}
such that $|G_i:G_{i-1}|$ is a prime for $i=1,2,\ldots,t$.

\begin{prop}\label{primes}
$\cP\subseteq\cF$. In particular, $\cF$ includes the class of finite solvable groups.
\end{prop}

Proposition~\ref{mth} implies that we can concentrate on nonabelian simple groups. So let us consider them starting from the smallest ones (here and further we use information from \cite{Atlas} and~\cite{GAP}). It is easily seen that the groups $A_5$ and $PSL_2(7)$ of orders $60$ and $168$ respectively contain solvable subgroups of prime indices, so they belong to the class $\cF$ by Proposition~\ref{primes}. The third smallest simple group $A_6$ of order $360=2^3\cdot3^2\cdot5$ contains no subgroups of prime index, so does not belong to $\cP$. However, it includes subgroups (obviously from $\cF$) of indices $6$, $10$, and $15$, therefore, for every divisor $a$ of $n=360$, either $a$ or $n/a$ divides the order of one of these subgroups. Thus, $A_6\in\cF$ in virtue of Proposition~\ref{mth}(ii).

Here we come to the simple group $G=SL_2(8)$ of all $2\times2$ matrices having the determinant~$1$ over the field of order $8$. This group has order $504$ and contains no subgroups of orders divisible by $12$ or $21$, so one cannot apply Proposition~\ref{mth}(ii) for two possible factorizations:
$$
504=21\cdot24\quad\mbox{and}\quad504=12\cdot42.
$$
It is also easily seen that a naive search does not work, because the number of possible factorizations of $G$ is huge. Therefore, in order to verify whether the suitable subsets exist, one need to apply more subtle arguments. For $G=SL_2(8)$, desired factorizations $G=AB$ exist, as Peter Mueller shown using computer calculations (see \cite{MOH}). His idea was to choose a first set $A$ as a product of two subgroups with $AA^{-1}$ as small as possible, and then find $B$ by solving the exact cover problem (see, e.g., \cite{Knuth}). We find such factorizations using a different approach. We choose `good' subgroups $H\subseteq A$, $K\subseteq B$ and use the well-known fact that $G$ is a disjoint union of the double cosets $HxK$. This significantly reduced the number of calculations required. Moreover, by the same method, we prove that fifteen smallest nonabelian simple groups  belong to $\cF$ (see Section~\ref{ssg}). Thus, the following assertion holds.

\begin{theorem}\label{thssg}
Suppose that the order of every nonabelian composition factor of a~finite group $G$ does not exceed $10\,000$. Then $G\in\cF$.
\end{theorem}

\section{Proof of Propositions~\ref{mth} and~\ref{primes}}\label{pth2}

Part (i) of Proposition~\ref{mth} is an immediate consequence of the following lemma.

\begin{lemma}\label{simple}
Let $N$ be a normal subgroup of a finite group $G$. If $N,G/N\in\cF$, then $G\in\cF$.
\end{lemma}

\begin{proof} Put $\ov{G}=G/N$ and set $|N|=k$, $|\ov{G}|=\ov{k}$. Take any positive integers $a_0,b_0,\ov{a},\ov{b}$ such that $a=a_0\ov{a}$, $b=b_0\ov{b}$ and $a_0b_0=k$, $\ov{a}\ov{b}=\ov{k}$, which is definitely always possible. By hypothesis, there are factorizations into subsets of $N$ and $\ov G$: $$N=A_0B_0\qquad\mbox{and}\qquad\ov{G}=\ov{A}\,\ov{B},$$ where $|A_0|=a_0$, $|B_0|=b_0$ and $|\ov{A}|=\ov{a}$, $|\ov{B}|=\ov{b}$.

Fix subsets $U$ and $V$ of $G$ such that $\ov{A}=\{Nu : u\in U\}$, $|U|=\ov{a}$ and $\ov{B}=\{Nv : v\in V\}$, $|V|=\ov{b}$. Then each coset $Nx$ of $\ov{G}$ can be  represented as $Nx=NuNv$, where $u\in U$ and $v\in V$. This yields

$$
G=\bigcup_{x\in G}Nx=\bigcup_{u\in U, v\in V}NuNv=\bigcup_{u\in U, v\in V}uNv=
$$
$$=\bigcup_{u\in U, v\in V}u(A_0B_0)v=\bigcup_{u\in U, v\in V}(uA_0)(B_0v).$$

Therefore, each $g\in G$ can be represented as $g=uyzv$, where $y\in A_0$, $z\in B_0$ and $u\in U$, $v\in V$. On the other hand, as readily seen, $|XY|\leq|X||Y|$ for all $X,Y\subseteq G$, so the equality $n=\ov{a}a_0b_0\ov{b}$ implies that $|UA_0|=\ov{a}a_0=a$ and $|B_0V|=b_0\ov{b}=b$. It follows that $G$ is factorized into the subsets $A=UA_0$ and $B=B_0V$ of sizes $a$ and $b$ respectively, as required.
\end{proof}

\begin{lemma}\label{second} Let a subgroup $H$ of a finite group $G$ belong to $\cF$. If $a\in\cD(G)$, then $a$ does not divide $|H|$.
\end{lemma}

\begin{proof} Set $|G|=n$, $b=n/a$, and $|H|=k$. Suppose to the contrary that $a$ divides~$k$, and put $c=k/a$. If $T$ is a right transversal of $H$ in $G$, then $|T|=|G:H|=b/c$. By hypothesis, there are subsets $A$ and $C$ of $H$ such that $H=AC$ and $|A|=a$, $|C|=c$. Put $B=CT$. Then
$$
G=HT=(AC)T=A(CT)=AB,
$$
and $|A|=a$, $|B|=b$, a contradiction.
\end{proof}

If $G$ is a group of the smallest order such $G\not\in\cF$, then every its proper subgroup $H$ belongs to $\cF$. Therefore, $a\in\cD(G)$ cannot divide $|H|$ by Lemma~\ref{second}. As was observed in the introduction, $a\in\cD(G)$ if and only if $|G|/a\in\cD(G)$, so $|G|/a$ cannot divide $|H|$ too. This proves part (ii) of Proposition~\ref{mth}.

Suppose now that $G$ belongs to $\cP$. Then it has a series of subgroups~(\ref{series}), where $|G_{i}:G_{i-1}|$ is a prime for $i=1,\ldots,t$. By induction on the length $t$ of the series, we may assume that $H=G_{t-1}\in\cF$. Since $|G:H|$ is a prime, $|H|$ is a multiple of $a$ or $|G|/a$ for every divisor $a$ of $|G|$. Lemma~\ref{second} yields $G\in\cF$, thus proving Proposition~\ref{primes}.

\medskip

Let us finish the preliminaries with the following useful (and straightforward) observation (see, e.g., \cite[Lemma~2]{Ber}). A group $G$ can be factorized into subsets of sizes $a$ and $b$ if and only if it has a factorization $G=AB$ with $|A|=a$, $|B|=b$, and $1\in A\cap B$. The latter factorization is called {\em normalized}.

\section{Small Simple Groups}\label{ssg}

In this section we deal with small nonabelian simple groups. The main source for us is {\em Atlas of Finite Groups} \cite{Atlas}, so we prefer further to use the single-letter notation for simple groups of Lie type, following Artin's convention \cite[Section~2 of~Preface]{Atlas}. For example, $L_n(q)$ means $PSL_n(q)$ and so on. For brevity, we also agree to say `simple' instead of `nonabelian simple'.

As explained in the introduction, $L_2(8)$ is the smallest simple group for which the question whether it belongs to $\cF$ is nontrivial. As we know \cite[p.~6]{Atlas}, the group $L_2(8)$ has order $n=504=2^3\cdot3^2\cdot7$ and contains, up to conjugation, three maximal subgroups $M_i,i=1,2,3,$ of orders $m_1=56=2^3\cdot7$, $m_2=18=2\cdot3^2$, $m_3=14=2\cdot7$, respectively. Due to Proposition~\ref{mth}, it suffices to show that $12,21\not\in\cD(L_2(8))$ (other possible numbers divide the orders of maximal subgroups).  However, there is no clear inductive way to exclude $12$ and $21$ from $\cD(L_2(8))$. So we make a little trick to get them in this case.

\medskip

Recall that given subgroups $H,K$ and an element $x$ of a group $G$, the set $HxK=\{hxk : h\in H, k\in K\}$ is called a {\em double coset} of subgroups $H$ and $K$ (in $G$) with representative~$x$. Distinct double cosets by the same subgroups are disjoint, so they partition the set $G$.

Suppose that the orders of subgroups $H$ and $K$ of a finite group $G$ are coprime. Then for every $x\in G$,
$$
|HxK|=|x^{-1}HxK|=\frac{|H^x||K|}{|H^x\cap K|}=|H||K|.
$$
Moreover, since
$$
G=\bigcup_{x\in T}HxK,
$$
it follows that $|G|=|H||T||K|$ for every full set of representatives $T$ of double cosets of $H$ and~$K$ in~$G$.

\medskip

Let now $n=ab$ be a factorization of the order of a group $G$. Assume that we have subgroups $H$ and $K$ of $G$ such that they are of coprime orders and $a=|H|a_0$, $b=|K|b_0$. Then in order to obtain a required factorization of $G$, it suffices to find subsets $A_0$ and $B_0$ of $G$ with the following properties $(*)$:
\begin{enumerate}
\item $A_0$ is a subset of a set of representatives of the right cosets of $H$ in $G$, and $|A_0|=a_0$;
\item $B_0$ is a subset of a set of representatives of the left cosets of $K$ in $G$, and $|B_0|=b_0$;
\item $T=A_0B_0$ is a full set of representatives of the double cosets of $H$ and $K$ in $G$.
\end{enumerate}

Indeed,
$$
G=\bigcup_{x\in T}HxK=\bigcup_{y\in A_0,z\in B_0}(Hy)(zK),
$$
and the same easy counting argument as in the proof of Lemma~\ref{simple} implies that $G=AB$ with $A=HA_0$, $B=B_0K$, and $|A|=a$, $|B|=b$.

\medskip

We are ready to return to small simple groups.

Let $G=L_2(8)$. First, set $a=21$, $b=24$. Take $H$ and $K$ to be a Sylow $7$-subgroup and Sylow $2$-subgroup of $G$ respectively, and fix a right transversal $S$ of $H$ and left transversal $Q$ of $K$. Then $a=3|H|$, $b=3|K|$ and $|S|=72$, $|Q|=63$. So in order to verify whether there exist subsets $A_0$ and $B_0$ satisfying $(*)$, one need to run over 3-element subsets of $S$ and 3-element subsets of~$Q$. Moreover, since we may suppose that $1=A_0\cap B_0$, it suffices to run over just 2-element subsets of $S$ and~$Q$. We made it applying GAP \cite{GAP}, and readily obtained suitable subsets $A_0$ and $B_0$ and, consequently, subsets $A$ and $B$ of sizes $21$ and $24$ that factorized~$G$ (the supplementary GAP-files are available at~\cite{Code} and~\cite{Logs}). If $a=12$ and $b=42$, then one can take $H$ to be a subgroup of order $6$ and $K$ to be a Sylow $7$-subgroup of~$G$. Similar calculations again give required subsets $A_0$, $B_0$, and hence $A$, $B$ (see~\cite{Logs}).

In fact, this approach allows to verify that every finite simple group $G$ of order less than $10\,000$ belongs to~$\cF$, thus proving Theorem~\ref{thssg}. Table~\ref{tabssg} below represents shortly our arguments. In this table, $G$ is one of the simple groups of order $n\leq10\,000$; the numbers $m_i$ are the orders of the maximal subgroups of $G$ (if there is a subgroup of prime index, then only its order is represented, cf. Proposition~\ref{primes}); the integers $a$ are possible elements of $\cD(G)$ (modulo Proposition~\ref{mth}(ii)) not exceeding $\sqrt{n}$; the numbers $h$ and $k$ are the coprime orders of subgroups $H$ and $K$ of $G$ respectively, with $h$ dividing $a$ and $k$ dividing $n/a$, such that $A=HA_0$ and $B=B_0K$ give a~factorization $G=AB$ (for details, see \cite{Logs}).

\medskip

\begin{tab}\label{tabssg}{\bfseries Factorizations of Simple Groups of Order Less Than 10\,000}
\medskip

{\small    \noindent\begin{tabular}{|l|l|l|c|c|}
  \hline
  $G$ & $n$ & $m_i$ & $a$ & $h,k$ \\
  \hline
$A_5$ & $60=2^2\cdot3\cdot5$ & $2^2\cdot3,\,\ldots$  & --- {\em (Prop.~\ref{primes})} & --- \\ \hline
$L_2(7)$ & $168=2^3\cdot3\cdot7$ & $2^3\cdot3,\,\ldots$ & --- {\em (Prop.~\ref{primes})} & --- \\ \hline
$A_6$ & $360=2^3\cdot3^2\cdot5$ & $2^2\cdot3\cdot5,\, 2^2\cdot3^2,\, 2^3\cdot3$ & --- {\em (Prop.~\ref{mth})} & --- \\ \hline
$L_2(8)$ & $504=2^3\cdot3^2\cdot7$ & $2^3\cdot7,\, 2\cdot3^2,\, 2\cdot7$ & $12=2^2\cdot3$ & $2\cdot3,\,7$ \\
 &  &  & $21=3\cdot7$ & $7,\,2^3$ \\ \hline
$L_2(11)$ & $660=2^2\cdot3\cdot5\cdot11$ & $2^2\cdot3\cdot5,\,\ldots$ & --- {\em (Prop.~\ref{primes})} & --- \\ \hline
$L_2(13)$ & $1\,092=2^2\cdot3\cdot7\cdot13$ & $2\cdot3\cdot13,\, 2\cdot7,\, 2^2\cdot3$ & $21=3\cdot7$ & $7,\,2\cdot13$ \\ \hline
$L_2(17)$ & $2\,448=2^4\cdot3^2\cdot17$ & $2^3\cdot17,\, 2^3\cdot3,\, 2\cdot3^2,\, 2^4$ & $48=2^4\cdot3$ & $2^4,\,17$ \\ \hline
$L_2(19)$ & $3\,420=2^2\cdot3^2\cdot5\cdot19$ & $3^2\cdot19,\, 2^2\cdot3\cdot5,$ & $36=2^2\cdot3^2$ & $2\cdot3^2,\,19$ \\
 &  & $2^2\cdot5,\, 2\cdot3^2$  & $38=2\cdot19$ & $19,\,2\cdot3^2$ \\
 &  &   & $45=3^2\cdot5$ & $3^2,\,19$ \\ \hline
$L_2(16)$ & $4\,080=2^4\cdot3\cdot5\cdot17$ & $2^4\cdot3\cdot5,\,\ldots$ & --- {\em (Prop.~\ref{primes})} & --- \\ \hline
$L_3(3)$ & $5\,616=2^4\cdot3^3\cdot13$ & $2^4\cdot3^3,\,\ldots$ & --- {\em (Prop.~\ref{primes})} & --- \\ \hline
$U_3(3)$ & $6\,048=2^5\cdot3^3\cdot7$ & $2^3\cdot3^3,\, 2^3\cdot3\cdot7,\, 2^5\cdot3$ & --- {\em (Prop.~\ref{mth})} & --- \\ \hline
$L_2(23)$ & $6\,072=2^3\cdot3\cdot11\cdot23$ & $11\cdot23,\, 2^3\cdot3,\, 2\cdot11$ & $44=2^2\cdot11$ & $2\cdot11,\,23$\\
 &  &  & $46=2\cdot23$ & $23,\,2\cdot11$\\
 &  &   & $66=2\cdot3\cdot11$ & $2\cdot11,\,23$\\
 &  &   & $69=3\cdot23$ & $23,\,2\cdot11$\\  \hline
$L_2(25)$ & $7\,800=2^3\cdot3\cdot5^2\cdot13$ & $2^2\cdot3\cdot5^2,\, 2^3\cdot3\cdot5,$ & $39=3\cdot13$ & $13,\,2^2\cdot5^2$\\
 &  & $2\cdot13,\, 2^3\cdot3$  &   &  \\ \hline
$M_{11}$ & $7\,920=2^4\cdot3^2\cdot5\cdot11$ & $2^4\cdot3^2\cdot5,\,\ldots$ & --- {\em (Prop.~\ref{primes})} & --- \\ \hline
$L_2(27)$ & $9\,828=2^2\cdot3^3\cdot7\cdot13$ & $3^3\cdot13,\, 2^2\cdot7,$ & $18=2\cdot3^2$   & $3^2,\,2\cdot13$ \\
 &  & $2\cdot13,\, 2^2\cdot3$ & $21=3\cdot7$   & $7,\,2\cdot13$ \\
 &  &  & $36=2^2\cdot3^2$  & $2^2\cdot3,\,13$ \\
 &  &  & $42=2\cdot3\cdot7$   & $2\cdot7,\,13$ \\
 &  &  & $52=2^2\cdot13$  & $2\cdot13,\,3^3$ \\
 &  &  & $54=2\cdot3^3$  & $3^3,\,2\cdot13$ \\
 &  &  &  $63=3^2\cdot7$ & $3^2,\,2\cdot13$ \\
 &  &  & $78=2\cdot3\cdot13$ & $2\cdot13,\,3^2$ \\
 &  &  & $91=7\cdot13$  & $13,\,3^3$ \\ \hline
\end{tabular}}
\end{tab}

\medskip

We conclude with some remarks on the program that we used for calculations. If the indices of $H$ and $K$ are sufficiently small (as in the case of the group $G=L_2(8)$), then one can find suitable sets $A_0$ and $B_0$ by running over all possible subsets of transversals $S$ and~$Q$. However, it would require too much time in larger groups, so our program runs through the possible subsets $A_0$ of $S$ and then find $B_0$ by solving exact cover problem (see \cite{Knuth}). In fact, we can find certain nontrivial factorizations (but not all of them) for simple groups of order greater that $10\,000$, even for the Mathieu group $G=M_{12}$ of order $95\,040=2^6\cdot3^3\cdot5\cdot11$. As one can check using \cite{Atlas} and Proposition~\ref{mth}, it suffices to exclude $a=135,\, 270,\, 297$ from $\cD(G)$. We verified that $135,297\not\in\cD(G)$, so it would be interesting to exclude the only possible element $270$ from~$\cD(G)$.

\end{document}